\documentclass[smallextended,envcountsect]{svjour3}
\smartqed
\usepackage[left=4.5cm, right=4.5cm, top=3.5cm, bottom=3.5cm]{geometry}
\usepackage{cite}
\usepackage[weather,alpine,misc,geometry]{ifsym}
\usepackage{mathptmx} 
\usepackage{amsmath}
\usepackage{amssymb}

\newcommand{\al}{\alpha}

\newcommand{\ga}{\gamma}

\newcommand{\de}{\delta}
\newcommand{\eps}{\varepsilon}
\newcommand{\bx}{\bar x}
\newcommand{\by}{\bar y}

\newcommand{\iv}{^{-1} }

\newcommand {\R} {\mathbb R}
\newcommand {\N} {\mathbb N}

\newcommand {\B} {\mathbb B}

\newcommand {\gph} {{\rm gph}\,}
\newcommand {\dom} {{\rm dom}\,}

\newcommand {\sd} {\partial}

\def\es{\emptyset}

\def\Fr{Fr\'echet}

\newcommand{\ang}[1]{\left\langle #1 \right\rangle}

\newcounter{mycount}


\usepackage[colorinlistoftodos]{todonotes}
\usepackage[weather,alpine,misc,geometry]{ifsym}

\renewcommand\theenumi{\roman{enumi}}

\usepackage{hyperref}
\hypersetup{
colorlinks=true,
linkcolor=blue,
citecolor=red,
urlcolor=blue}
\begin{document}
\title{Lyusternik-Graves Theorem for H\"older Metric Regularity
}

\author{Nguyen Duy Cuong
}
\institute{
Nguyen Duy Cuong 	(\Letter)
\\
{Department of Mathematics, College of Natural Sciences, Can Tho University, Vietnam}\\
Email: ndcuong@ctu.edu.vn
}


\date{Received: date / Accepted: date}
\maketitle

\begin{abstract}
The paper extends the well-known Lyusternik-Graves theorem for set-valued mappings to the H\"older framework, offers an affirmative answer to an open problem proposed by Dontchev  and improves recent results of He and Ng.
Primal and dual necessary and sufficient conditions for H\"older metric regularity are established.
The  results are applied to convergence analysis of a Newton-type method.
Some open problems for future research are also discussed.

\end{abstract}




\section{Introduction}\label{sect1}
It is well known that many important problems in variational analysis and optimization, cf. \cite{DonRoc14,Mor06.2,Iof17,RocWet98}, can be modelled by 
the  generalized equation
\begin{gather}\label{GE}
 F(x)\ni y
\end{gather}
where $F:X\rightrightarrows Y$ is a set-valued mapping between metric spaces.
When the mapping is  single-valued,  inclusion \eqref{GE} reduces to a conventional equation, but more broadly it can express a mixture of inequalities and equalities.
Relation \eqref{GE} can represent  a variational inequality or a system of optimality conditions.
An important issue in investigating a generalized equation is to study the behavior of the solution set $F^{-1}(y)$  with respect to perturbations in $y$, and this can often be expressed in terms of the property called  `metric regularity'.
The property has its roots in the classical results by Banach and plays a central role in variational analysis  both theoretically and numerically  \cite{DonRoc14,Mor06.1,RocWet98,Iof17}.

\begin{definition}\label{D1.3}
Let $X$ and $Y$ be metric spaces, $F:X\rightrightarrows Y$, and $(\bx,\by)\in \gph F$.
The mapping $F$ is metrically regular at $(\bx,\by)$ if there exist $\tau>0$ and $\de>0$ such that
\begin{align}\label{D1.1-4}
\tau d(x,F\iv(y))\le d(y,F(x))
\end{align}
for all $x\in B_{\de}(\bx)$ and $y\in B_{\de}(\by)$.
The supremum of $\tau$ such that \eqref{D1.1-4} holds for some $\de>0$ is called the  modulus of regularity of $F$ at $(\bx,\by)$ and denoted by 
\rm{rg}$F(\bx,\by)$.
\end{definition}

It should be noted that \rm{rg}$F(\bx,\by)=$1/\rm{reg}$(F;\bx|\by)$ where  \rm{reg}$(F;\bx|\by)$ is the modulus of regularity employed in \cite{DonRoc14,DonRoc04,DonLewRoc03}.
This quantity enables one to check how large a perturbation can be before a `good behavior' of the solution mapping breaks down.

Metric regularity necessitates uniformity in estimates involving local perturbations of both $\bx$ and $ \by$.
The term $d(y,F(x))$ measures the residual when $y\notin F(x)$. Strictly speaking, inequality \eqref{D1.1-4} provides an estimate of how distant a point $x$ is from being a solution to the generalized equation \eqref{GE}.
Computing a residual is much easier than finding a solution to the generalized equation. 
Such an estimate is crucial for numerous optimization problems, especially for computational purposes.

The study of the H\"older metric regularity has attracted considerable attention due to the fact that the conventional (linear) metric regularity fails in many practical situations.
The number of publications dedicated to studying H\"older metric regularity is large, see \cite{CuoKru21.2,FraQui12,Chu15,ChuKim16,YenYaoKie08,LeePha22} and the references therein.

\begin{definition}\label{D1.1}
Let $X,Y$ be metric spaces, $F:X\rightrightarrows Y$, $(\bx,\by)\in \gph F$,  and $q>0$.
The mapping $F$ is metrically regular of order $q$ at $(\bx,\by)$  if there exist $\tau>0$ and  $\de>0$ such that
\begin{align}\label{D1.1-1}
\tau d(x,F\iv(y))\le d^q(y,F(x))
\end{align}
for all $x\in B_{\de}(\bx)$ and $y\in B_{\de}(\by)$.
The supremum of $\tau$ such that \eqref{D1.1-1} holds for some $\de>0$ is called the modulus of regularity of order $q$ of $F$ at $(\bx,\by)$ and denoted by 
\rm{rg}$^qF(\bx,\by)$.
\end{definition}

The next definition recalls the concept of H\"older continuity of set-valued mappings \cite{KlaKruKum12,YenYaoKie08}.

\begin{definition}\label{D1.5}
Let $X,Y$ be metric spaces, $\Phi:X\rightrightarrows Y$, $(\bx,\by)\in\gph F$,  and $q>0$.
The mapping $\Phi$ is H\"older continuous of order $q$ at $(\bx,\by)$  if there exist $\tau>0$ and  $\de>0$ such that
\begin{align*}
d^q(y,\Phi(x))\le \tau d(x,x')
\end{align*}	
for all $x,x'\in B_{\de}(\bx)$ and $y\in \Phi(x')\cap B_{\de}(\by)$.
The infimum of $\tau$ such that the above inequality holds for some $\de>0$ is called the modulus of H\"older continuity and denoted by ${\rm{lip}}^q \Phi(\bx,\by)$.
\end{definition}	

When  $\Phi$ is single-valued, the property in Definition~\ref{D1.5} reduces to the conventional  H\"older continuity \cite{HeNg18,Kir98}, and the  modulus is denoted by ${\rm{lip}}^q \Phi(\bx)$.
The absence of the above  properties is characterized by ${\rm{rg}}^qF(\bx,\by)=0$ and ${\rm{lip}}^q\Phi(\bx,\by)=+\infty$, respectively.
In the case $q=1$, the properties  reduce to the convetional metric regularity and Aubin property  \cite{DonRoc14,DonLewRoc03}.

The next statement is straighforward.
\begin{proposition}
Let $X,Y$ be metric spaces, $F:X\rightrightarrows Y$, $(\bx,\by)\in \gph F$,  and $q>0$.
The mapping $F$ is metrically regular of order $q$ at $(\bx,\by)$  if and only if $F\iv$ is H\"older continuous of order $\frac{1}{q}$ at $(\by,\bx)$.
Moreover,
\begin{gather*}
{\rm{rg}}^qF(\bx,\by)=({\rm{lip}}^{\frac{1}{q}} F\iv(\by,\bx))^{-q}.
\end{gather*}	
\end{proposition}	

The property in Definition~\ref{D1.1} does not change if one imposes an upper bound on the right-hand side of \eqref{D1.1-1}; cf.
\cite[Excercise~2.12]{Iof17} and \cite[Proposition~1.137]{Pen13}.
\begin{proposition}\label{P3.1}
Let $X$ and $Y$ be metric spaces, $F:X\rightrightarrows Y$, $(\bx,\by)\in \gph F$, and $q>0$.
The mapping $F$ is metrically regular of order $q$ at $(\bx,\by)$ if and only if there exist $\tau>0$, $\de>0$, and $\mu>0$ such that inequality \eqref{D1.1-1} holds for all $x\in B_{\de}(\bx)$ and $y\in B_{\de}(\by)$ with $d^q(y,F(x))<\tau\mu$.
\end{proposition}	

If  inequality \eqref{D1.1-1} holds for all $x\in B_{\de}(\bx)$ and $y\in B_{\de}(\by)$ with $d^q(y,F(x))<\tau\mu$, then we often say that $F$ is metrically regular of order $q$ at $(\bx,\by)$ with  $\tau$, $\delta$ and $\mu$.
The condition $d^q(y,F(x))<\tau\mu$ in Proposition~\ref{P3.1} can be replaced by $d^q(y,F(x))<\mu$.
However, we prefer to keep the current form since it results in `neater' statements in Section~\ref{S4}.
In any case, adding such a condition does not affect the constant $\tau$ in Definition~\ref{D1.1}, but  can have an effect on the value of $\de$.
We consider in this paper the H\"older metric regularity for the case $q\in(0,1]$, although some results are also valid when $q>1$.

In the current paper, we study single-valued additive perturbations of the left-hand side of \eqref{GE}; cf. \cite{DonRoc14}. 
Set-valued perturbations were considered  in \cite{AdlNgaVu17,DonFra12,HeXu22,Iof01}.
What effect do such perturbations have on the metric regularity propery? 
The question is answered by the fundamental estimation arising from the works of Lyusternik \cite{Lyu34} and Graves \cite{Gra50}.
It shows that metric regularity of a set-valued mapping is preserved if the perturbation function  is Lipschitz continuous with  a sufficiently small Lipschitz constant.
Interested readers are referred to 
\cite{DonRoc14,Iof17,DmiMilOsm80,Iof00,DonFra10,DonFra12} and the references therein.
\begin{theorem}\label{T1.2}
Let $X$ be a complete metric space, $Y$ be a linear space with a shift-invariant metric, $F:X\rightrightarrows Y$, $(\bx,\by)\in\gph F$, $f: X\rightarrow Y$ with $f(\bx)=0$, and $\gph F$ be closed near
$(\bx,\by)$.
Then
\begin{align*}
 {\rm{rg}}(F+f)(\bx,\by)\ge  {\rm{rg}}F(\bx,\by)-{\rm{lip}}f(\bx).
\end{align*}	
\end{theorem}	

The right-hand side quantity in the above inequality can be undefined when both moduli equal plus infinity.
To address this situation, we employ in the current paper the convention that $(+\infty)-(+\infty)=0.$

In the paper \cite{Don15}, Dontchev posed an open question about a possible generalization of Theorem~\ref{T1.2} to the H\"older setting.
Recently, He and Ng \cite[Theorem~3]{HeNg18} have proved a H\"older version of Theorem~\ref{T1.2} by establishing a relation for the corresponding constants in the definitions of metric regularity and H\"older continuity.
In \cite{FraQui12,Xu22}, the authors have provided some  primal estimations for  the modulus of H\"older regularity.
In the current paper, we establish an improved version of  \cite[Theorem~3]{HeNg18}.
The results are applied to convergence analysis  of a Newton-type method  enhancing  \cite[Theorem~15.1]{Don21}.

The paper is organized as follows. 
The next Section~\ref{S2} provides some preliminary results used throughout the paper. 
Section~\ref{S4} studies slope  and coderivative necessary and sufficient conditions for H\"older metric regularity.
We establish in Section~\ref{S6} a H\"older version of the extended Lyusternik-Graves theorem. 
The results are applied in Section~\ref{S5} to convergence analysis of a Newton-type method.
The final Section~\ref{S7} proposes some open problems for further research.
\section{Preliminaries}\label{S2}
Our basic notation is standard, see, e.g., \cite{Mor06.1,RocWet98,DonRoc14}.
Throughout the paper, if not explicitly stated otherwise, $X$ and $Y$ are  metric spaces.
Products of metric or normed spaces are assumed to be equipped with the maximum distance or norm.
The topological dual of a normed space $X$ is denoted by $X^*$, while $\langle\cdot,\cdot\rangle$ denotes the bilinear form defining the pairing between the two spaces.
In a primal space, the open and closed balls with center $x$ and radius $\delta>0$ are denoted, respectively, by $B_\delta(x)$ and $\overline{B}_\de(x)$, while $\B$ and $\overline{\B}$ stand for, respectively, the open and closed unit balls.
The open unit ball in the dual space is denoted by $\B^*$.
A set $\Omega$ is said to be closed near $\bx\in\Omega$ if there exists a $\de>0$ such that  $\Omega\cap\overline B_\de(\bx)$ is closed.
Symbols $\R$, $\R_+$ and $\N$ stand for the real line, the set of all nonnegative reals, and the set of all nonnegative integers, respectively.

A metric $d$ on a vector space $X$ is called shift-invariant if 
$d(x'+z,x+z)=d(x,x')$ for all $x',x,z\in X$.
For subsets $A,B$ of a metric space $X$, the excess of $A$ beyond $B$ is defined by
$e(A,B):=\sup_{x\in A}d(x,B)$
with the convention that $e(\emptyset,B):=0$ when $B\ne\emptyset$ and $+\infty$ otherwise.

Let $\{x_k\}_{k\in \N}$  be a sequence in a normed space $X$  converging to a point $\bx\in X$.
It is said to converge quadratically to $\bx$ if 
there exist $\gamma>0$ and $k_0\in\N$ such that
$\|x_{k+1}-\bx\|\le \gamma \|x_k-\bx\|^2$ for all $k\ge k_0$.

Let $X$ be a normed space, $\Omega\subset X$, and $f:X\to\R\cup\{+\infty\}$.
The Fr\'echet normal cone to $\Omega$ at $\bx\in \Omega$ and the Fr\'echet subdifferential of $f$ at $\bar x\in\dom f:=\{x\in X\mid f(x)< +\infty\}$ are defined, respectively, by
\begin{gather*}
N_{\Omega}(\bx):= \left\{x^\ast\in X^\ast\mid
\limsup_{\Omega\ni x\to\bar x,\,x\ne \bx} \frac {\langle x^\ast,x-\bx\rangle}
{\|x-\bx\|} \le 0 \right\},\\	
\partial f(\bar x):=\left\{x^*\in X^*\mid \liminf_{\substack{x\to \bar x,\,x\ne\bx}} \dfrac{f(x)-f(\bar x)-\langle x^*,x-\bar x\rangle}{\|x-\bar x\|}\ge 0\right\}.
\end{gather*}
By convention, we set $N_{\Omega}(\bx) :=\es$ if $\bx\notin \Omega$ and $\partial{f}(\bx):=\es$ if $\bx\notin\dom f$.
If $\Omega$ and $f$ are convex, the aforementioned concepts reduce to  the normal cone and subdifferential in the sense of convex analysis.
If $f$ is Fr\'echet differentiable with a derivative $\nabla f(\bx)$, then $\partial f(\bar x)=\{\nabla f(\bx)\}$.

The Fr\'echet coderivative of a set-valued mapping  $F:X\rightrightarrows Y$ between normed spaces at $(\bx,\by)\in\gph F$ is a set-valued mapping $D^*F(\bx,\by):Y^*\rightrightarrows X^*$ defined for any $y^*\in Y^*$ by
\begin{align}\label{coder}
D^*F(\bx,\by)(y^*):=\{x^*\in X^*\mid (x^*,-y^*)\in N_{\gph F}(\bx,\by)\}.
\end{align}

The following results are well known  \cite{Kru03,Mor06.1,Zal02}.

\begin{lemma}\label{L2.4}
Let $X$ be a normed space, $f:X\to\R\cup\{+\infty\}$, and $\bx\in\dom f$.
The following statements hold.
\begin{enumerate}
\item	
If $x$ is a point of local minimum of $f$, then $0\in\sd f(\bx)$.
\item 
$\partial(\lambda f)(\bx)=\lambda\partial f(\bx)$ for any $\lambda>0$.
\item
$\sd\|\cdot\|(0)=\{x^*\in X^*\mid
\|x^*\|\le 1\}$.
\item
$\sd\|\cdot\|(x)=\{x^*\in X^*\mid \langle x^*,x\rangle=\|x\|\;\;
\text{and}
\;\;
\|x^*\|= 1\},
\;\;
x\ne 0$.
\end{enumerate}
\end{lemma}

Let $X$ be a metric space,  $f:X\rightarrow\R\cup\{+\infty\}$. 
The slope  \cite{NgaThe08,Iof00,AzeCorLuc02,Kru15} of $f$ at $x\in\dom f$ is defined by
\begin{align*}
|\nabla f|(x):=\limsup_{u\rightarrow x,u\ne x}\dfrac{ [f(x)-f(u)]_+}{d(x,u)} 
\end{align*}
where $\al_+:=\max\{0,\al\}$ for any $\al\in\R$.
If $x\notin\dom f$, we set ${|\nabla f|(x):=+\infty}$.

The next statement offers chain rules for slopes  \cite[Lemma~1.1]{CuoKru21} and Fr\'echet subdifferentials \cite[Proposition~2.1]{CuoKru20.2}.
\begin{lemma}\label{L2.6}
Let $X$ be a metric space, $f: X\rightarrow\R\cup\{+\infty\}$, $\bx\in\dom f$ with $f(\bx)>0$, and $q>0$.
The following statements hold.
\begin{enumerate}
\item 
$|\nabla f^q|(\bx)=qf^{q-1}(\bx)|\nabla f|(\bx)$.
\item	
If $X$ is a normed space, then
$\partial f^q(\bx)=qf^{q-1}(\bx)\partial f(\bx)$.
\end{enumerate}
\end{lemma}	

The other fundamental tools for our analysis are the contraction mapping principle for set-valued mappings  \cite{DonHag94.1}, the Ekeland variational principle \cite{Eke74}, and subdifferential sum rules \cite{Zal02,Kru03,Fab89}.

\begin{lemma}\label{L2.1}
Let $X$ be a complete metric space, $\Phi:X\rightrightarrows X$, $x\in X$, $\theta\in(0,1)$, and $\de>0$.
Suppose that the following conditions are satisfied:
\begin{enumerate}
\item 
$\gph\Phi\cap [\overline B_\de(x)\times \overline B_\de(x)]$ is closed;
\item 
$d(x,\Phi(x))<\de(1-\theta)$;
\item 
$e(\Phi(u)\cap B_\de(x),\Phi(v))\le\theta d(u,v)$ for all $u,v\in\overline B_\de(x)$.
\end{enumerate}	
Then,  there exists an $\hat x\in \overline B_\de(x)$ with $\hat x\in\Phi(\hat x)$.
If $\Phi$ is single-valued, then $\hat x$ is the unique fixed point in $\overline B_\de(x)$.
\end{lemma}	

\begin{lemma}\label{EVP}
Let $X$ be a complete metric space, $f: X\to \mathbb{R} \cup \{ +\infty\}$ be lower semicontinuous,
$x\in X$, $\varepsilon>0$
and $\lambda>0$. 
If $f(x)<\inf_{X} f+\varepsilon$,
then there exists an $\hat x\in X$ such that
\begin{enumerate}
\item
$d(\hat{x},x)<\lambda$;
\item
$f(\hat{x})\le f(x)$;
\item
$f(u)+(\varepsilon/\lambda)d(u,\hat{x})\ge f(\hat{x})$ for all $u\in X.$
\end{enumerate}
\end{lemma}

\begin{lemma}
\label{SR}
Let $X$ be a normed space, $f_1,f_2:X\to\R \cup\{+\infty\}$, and $\bx\in\dom f_1\cap\dom f_2$.
\begin{enumerate}
\item
Suppose $f_1$ and $f_2$ are convex, and $f_1$ be continuous at a point in $\dom f_2$.
Then
$$\partial(f_1+f_2)(\bx)=\sd f_1(\bx)+\partial f_2(\bx).$$
\item
Suppose $X$ is Asplund,
$f_1$ is Lipschitz continuous and $f_2$ is lower semicontinuous in a neighbourhood of $\bx$.
Then, for any $x^*\in\partial(f_1+f_2)(\bx)$ and $\varepsilon>0$, there exist $x_1,x_2\in X$ with $\|x_i-\bx\|<\varepsilon$, $|f_i(x_i)-f_i(\bx)|<\varepsilon$ $(i=1,2)$ such that
$$x^*\in\partial f_1(x_1) +\partial f_2(x_2)+\varepsilon\B^\ast.$$
\end{enumerate}
\end{lemma}

Recall that a Banach space is Asplund if every continuous convex function on an open convex set is Fr\'echet differentiable on a dense subset \cite{Phe93}, or equivalently, if the dual of each its separable subspace is separable.
We refer the reader to \cite{Phe93,Mor06.1,BorZhu05} for discussions about and characterizations of Asplund spaces.
All reflexive, particularly, all finite dimensional Banach spaces are Asplund.
\section{Necessary and sufficient conditions for H\"older metric regularity}\label{S4}
Along with the standard maximum {distance} on $X\times Y$, we also use a {metric} depending on a parameter $\ga>0$ defined by
\begin{gather}\label{pdist}
d_\ga((u_1,v_1),(u_2,v_2)) :=\max\left\{d(u_1,u_2),\ga d(v_1,v_2)\right\}
\end{gather}
for any $u_1,u_2\in X,\;v_1,v_2\in Y$.
When $X$, $Y$ are normed spaces, the distance \eqref{pdist} yields the definition of the parametric norm
\begin{gather}\notag
\|(x,y)\|_{\ga}:=\max\{\|x\|,{\ga}\|y\|\},\quad
x\in X,\;y\in Y,
\end{gather}
and the corresponding dual norm
\begin{align}\label{dnorm}
\|(x^*,y^*)\|_{\ga}=\|x^*\|+\ga\iv\|y^*\|,\quad
x^*\in X^*,\;y^*\in Y^*.
\end{align}

\begin{theorem}\label{T3.1}
Let $X, Y$ be metric spaces,  $F: X\rightrightarrows Y$, $(\bx,\by)\in\gph F$, and $q\in (0,1]$.
\begin{enumerate}
\item
Suppose $X$ and $Y$ are complete, and $\gph F$ is closed.
If there exist $\tau>0$, $\de>0$, $\mu>0$, and $\gamma>0$ such that  
\begin{align}\label{R3.1-2}
\limsup_{\substack{
u\to x,\,v\to z,\;(u,v)\in\gph F\\(u,v)\ne (x,z),\,d(u,\bx)<\de+\mu,\,d(v,y)<(\tau\mu)^{\frac{1}{q}}}}
{\dfrac{d^q(z,y)-d^q(v,y)}{d_\gamma((u,v),(x,z))}}\ge\tau
\end{align}	
for all  $x\in B_{\de+\mu}(\bx)$, $y\in B_{\de}(\by)$ with $x\notin F\iv(y)$, and $z\in F(x)$ with 
$d(y,z)<(\tau\mu)^{\frac{1}{q}}$, then $F$ is metrically regular of order $q$ at $(\bx,\by)$ with  $\tau$, $\delta$ and $\mu$.
\item
Suppose $X,Y$ are normed spaces, and $\gph F$ is convex.
If $F$ is metrically regular of order $q$  at $(\bx,\by)$ with some $\tau>0$, $\delta>0$ and $\mu>0$, then 
\begin{align}\label{R3.1-3}
\limsup_{\substack{
u\to x,\,v\to z,\;(u,v)\in\gph F\\(u,v)\ne (x,z),\,\|u-\bx\|<\de+\mu,\,\|v-y\|<(\tau\mu)^{\frac{1}{q}}}}
{\dfrac{\|z-y\|^q-\|v-y\|^q}{\|(u-z,v-z)\|_\gamma}}\ge\tau
\end{align}	
for $\gamma:=\tau\iv$, and all  $x\in B_{\de}(\bx)$, $y\in B_{\de}(\by)$ with $x\notin F\iv(y)$, and $z\in F(x)$ with $\|z-y\|<\min\{(\tau\mu)^{\frac{1}{q}},1\}$.
\end{enumerate}
\end{theorem}

\begin{proof}
\begin{enumerate}
\item
Let $\tau>0$, $\de>0$, $\mu>0$, and $\gamma>0$.
Suppose $F$ is not metrically regular of order $q$ at $(\bx,\by)$ with $\tau$, $\delta$, and $\mu$.
By Proposition~\ref{P3.1}, there exist $x\in B_{\de}(\bx)$ and $y\in B_{\de}(\by)$ such that
$d^q(y,F(x))<\tau\mu_0$ with $\mu_0:=\min\{ d(x,F\iv(y)),\mu\}$.
Choose a number $\varepsilon$ such that $d^q(y,F(x))<\varepsilon<\tau\mu_0$, and a point $z\in F(x)$ such that $d^q(z,y)<\varepsilon$.
Let $\psi_y:X\times Y\rightarrow\R_+\cup\{+\infty\}$ be defined by 
\begin{gather}\label{psi}
\psi_y(u,v):=d^q(v,y)+i_{\gph F}(u,v),
\quad
u\in X,\;v\in Y.
\end{gather}
In view of the closedness of $\gph F$, the indicator function in \eqref{psi} is lower semicontinuous, and consequently, $\psi_y$ is lower semicontinuous on $\overline{B}_{\de+\mu}(\bx)\times{\overline{B}_{(\tau\mu)^{1/q}}(\by)}$.
Besides,
\sloppy
\begin{align*}
\psi_y(x,z)=d^q(z,y)+i_{\gph F}(x,z)=d^q(z,y)< \inf_{\overline{B}_{\de+\mu}(\bx)\times {\overline{B}_{(\tau\mu)^{1/q}}(\by)}}\psi_y+\varepsilon.
\end{align*}
Applying the Ekeland variational principle  (Lemma~\ref{EVP}) to the restriction of $\psi_y$ to the complete metric space $\overline{B}_{\de+\mu}(\bx)\times {\overline{B}_{(\tau\mu)^{1/q}}(y)}$ with the metric \eqref{pdist}, we can find a point $(\hat{x},\hat{z})\in \overline{B}_{\de+\mu}(\bx)\times {\overline{B}_{(\tau\mu)^{1/q}}(y)}$ such that
\begin{gather}\label{P1-3}
d_\ga((\hat{x},\hat{z}),(x,z))<\mu_0,\\ \label{P1-4}
\psi_y(\hat{x},\hat{z})\le\psi_y(x,z),
\\\label{P1-5}
\psi_y(\hat{x},\hat{z})\le\psi_y(u,v) +(\varepsilon/\mu_0) d_\ga((u,v),(\hat{x},\hat{z}))
\end{gather}
for all $(u,v)\in \overline{B}_{\de+\mu}(\bx)\times {\overline{B}_{(\tau\mu)^{1/q}}(y)}$.
It is clear from \eqref{P1-4} that $(\hat x,\hat z)\in\gph F$.
By \eqref{P1-3}, $d(\hat x,x)<d(x,F\iv(y))$.
Hence,  $\hat{x}\notin F\iv(y)$ and $\hat{z}\ne y$.
Besides,
\begin{gather*}
d(\hat{x},\bx)\le d(\hat x,x)+d(x,\bx)<\de+\mu,\;\;
d(\hat z,y)\le d(z,y)<\varepsilon^{\frac{1}{q}}<(\tau\mu)^{\frac{1}{q}}.
\end{gather*}
It follows from \eqref{P1-5} that
\begin{align*}
\sup_{\substack{(u,v)\in\gph F,\, (u,v)\ne(\hat{x},\hat{z}),\\
d(u,\bx)<\de+\mu,\,d(v,y)<(\tau\mu)^{\frac{1}{q}}}}
\dfrac{d^q(\hat z,y)-d^q(v,y)}{d_\ga((u,v),(\hat{x},\hat{z}))} \le\dfrac{\eps}{\mu_0}<\tau.
\end{align*}
The last estimate contradicts \eqref{R3.1-2}.
\item
Suppose $F$ is metrically regular of order $q$ at $(\bx,\by)$ with some $\tau>0$, $\delta>0$ and $\mu>0$.
By Proposition~\ref{D1.1}, inequality \eqref{D1.1-1} holds for all $x\in B_{\de}(\bx)$, $y\in B_{\de}(\by)$, and $d^q(y,F(x))<\tau\mu$.
Let $x\in B_{\de}(\bx)$, $y\in B_{\de}(\bx)$ with $x\notin F\iv(y)$, $z\in F(x)$ with $\|z-y\|<\min\{(\tau\mu)^{\frac{1}{q}},1\}$, $\eta>1$, and $\gamma:=\tau\iv$.
One can find
a $\xi\in(1,\eta)$ and a point $\hat{x}\in F\iv(y)$ such that
$\xi \|y-z\|^q<\tau\mu$, and $\tau{\|x-\hat{x}\|}<\xi\|z-y\|^q.$
Thus, $(\hat{x},y)\in\gph F$, $(\hat{x},y)\ne(x,z)$,
\begin{gather*}
\|\hat x-\bx\|\le \|\hat x-x\|+\|x-\bx\|<\tau\iv\xi \|z-y\|^q+\de<\de+\mu,
\end{gather*}
and
\begin{align*}
\|(x-\hat x,z-y)\|_\gamma
&=\max\{
\|x-\hat{x}\|,\ga \|z-y\|\}\\
&\le\tau\iv\max\{\xi,1\}\|z-y\|^q=\tau\iv\xi\|z-y\|^q.
\end{align*}
Hence,
\begin{align*}
\sup_{\substack{(u,v)\in\gph F,\,(u,v)\ne (x,z)\\
\|u-\bx\|<\de+\mu,\,\|v-y\|<(\tau\mu)^{\frac{1}{q}}}}
\dfrac{\|z-y\|^q-\|v-y\|^q}{\|(u-x,v-z)\|_\gamma}
\ge
\dfrac{\|z-y\|^q}{\|(\hat x-x,y-z)\|_\gamma}
\ge\tau\xi\iv>\tau\eta\iv.
\end{align*}
Letting $\eta\downarrow 1$, we arrive at \eqref{R3.1-3}.
\qed\end{enumerate}
\end{proof}

\begin{remark}
\begin{enumerate}
\item	
In the case $q=1$, part (i) of Theorem~\ref{T3.1} improves \cite[Proposition 5.5(ii)]{CuoKru21.4} and can be seen as a quantitative version of the first part of \cite[Theorem~3.13]{Iof17},  while part (ii) recaptures \cite[Proposition 5.5(i)]{CuoKru21.4}.
\item
The only difference between the expressions in the left-hand sides of \eqref{R3.1-2} and  \eqref{R3.1-3} is that
the first one is computed on metric spaces, while the second one is calculated on normed spaces.
The two expressions are the  slope at $(x,z)$ of the restriction of the function $\psi_y$, given by \eqref{psi}, to $\gph F\cap[B_{\de+\mu}(\bx)\times B_{(\tau\mu)^{1/q}}(y)]$.
\item
The completeness and closedness assumptions in  Theorem~\ref{T3.1}(i)  can be weakened: it suffices to require that $\gph F\cap [\overline{B}_{\de+\mu}(\bx)\times \overline{B}_{(\tau\mu)^{1/q}}(y)]$ is complete.
\end{enumerate}	
\end{remark}	

\begin{theorem}\label{P5.6}
Let $X,Y$ be normed spaces, $F: X\rightrightarrows Y$, $(\bx,\by)\in\gph F$, and $q\in (0,1]$.
\begin{enumerate}	
\item 	
Suppose $X$ and $Y$ are Asplund, and $\gph F$ is closed.
If there exist  $\tau>0$,  $\de>0$, $\mu>0$, $\eta>0$, and $\al\in(0,1)$ such that
\begin{align}\label{C3.3-3}
q\|z-y\|^{q-1} d(0,D^*F(x,z)(y^*))\ge \tau
\end{align}
for all $x\in B_{\de+\mu}(\bx)$, $y\in B_{\de}(\by)$ with $x\notin F\iv(y)$, and  $z\in F(x)$ with $\|z-y\|<(\tau\mu)^{\frac{1}{q}}$,  $y^*,z^*\in Y^*$ with
\begin{gather*}
 \|z^*\|=1,\; \langle z^*,z-y\rangle>\al\|z-y\|,\;  q\|z-y\|^{q-1}\|y^*-z^*\|<\eta,
\end{gather*}	
 then  $F$ is metrically regular of order $q$ at $(\bx,\by)$ with $\tau$, $\delta$ and $\mu$.
\item 	
Suppose $\gph F$ is convex.
If $F$ is metrically regular of order $q$ at $(\bx,\by)$ with some $\tau>0$, $\delta>0$ and $\mu>0$, then 
\begin{gather*}
q\|z-y\|^{q-1}	d(0,D^*F(x,z)(y^*))\ge\tau(1-\eta)
\end{gather*}
for all $\eta\in(0,1)$, $x\in B_{\de}(\bx)$, $y\in B_{\de}(\by)$ with $x\notin F\iv(y)$,  and $z\in F(x)$ with $\|z-y\|<\min\{(\tau\mu)^{\frac{1}{q}},1\}$,  $y^*,z^*\in Y^*$ satisfying 
\begin{align*}
	\|z^*\|=1,\;\; \langle z^*,z-y\rangle=\|z-y\|,\;\;
	q\|z-y\|^{q-1}\|y^*-z^*\|<\eta.
\end{align*}	
\end{enumerate}
\end{theorem}

\begin{proof}
\begin{enumerate}	
\item	
Let $\tau>0$, $\de>0$, $\mu>0$, $\eta>0$, 
$\al\in (0,1)$, $\gamma:=\tau\iv\eta$, and $\hat\tau\in(0,\tau)$.
Suppose $F$ is not metrically regular of order $q$  at $(\bx,\by)$ with $\tau$, $\de$, and $\mu$.
By Theorem~\ref{T3.1}(i), there exist $x\in B_{\de+\mu}(\bx)$, $y\in B_{\de}(\by)$ with $x\notin F\iv(y)$, $z\in F(x)$ with $\|z-y\|<(\hat\tau\mu)^{\frac{1}{q}}$, and  $\tau'\in(0,\hat\tau)$ such that
\begin{align*}
\|z-y\|^q-\|v-y\|^q\le\tau'\|(u-x,v-z)\|_\ga
\end{align*}
for all $(u,v)\in\gph F\cap [B_{\de+\mu}(\bx)\times B_{(\hat\tau\mu)^{1/q}}(y)]$.
In other words, $(x,z)$ is a local minimizer of the function
\begin{align}\label{P5P1}
(u,v)\mapsto\psi_y(u,v)+\tau'\|(u-x,v-z)\|_\ga,
\end{align}
where the function $\psi_y$ is defined by \eqref{psi}. 
By Lemma~\ref{L2.4}(i), its \Fr\  subdifferential at this point contains 0.
Observe that \eqref{P5P1} is the sum of the function $\psi_y$ and the Lipschitz continuous convex function $(u,v)\mapsto
\tau'\|(u-x,v-z)\|_{\ga}.$
In view of Lemma~\ref{L2.4}(iii) and (iv),  all subgradients $(u^*,v^*)$ of the latter function at any points satisfy
$\|(u^*,v^*)\|_{\ga}\le\tau'.$
Let $\varepsilon>0$ be such that
\begin{align*}
\varepsilon<\min\left\{\de+\mu-\|x-\bx\|, (\tau\mu)^{\frac{1}{q}}-\|z-y\|,
\hat\tau-\tau',\frac{1}{2}d(x,F\iv(y))\right\}.
\end{align*}
By Lemma~\ref{SR}(ii), there exist points $x'\in B_\varepsilon({x}),z'\in B_\varepsilon({z})$ with $(x',z')\in\gph F$, and $(\hat x^*,\hat z^*)\in \sd\psi_y(x',z')$ such that
$\|(\hat x^*,\hat z^*)\|_{\ga}<\tau'+\varepsilon.$
It is straighforward from the choice of $\varepsilon$ that
$x'\in{B_{\delta+\mu}(\bx)}$, $x'\notin F\iv(y)$, $\|z'-y\|<(\tau\mu)^{\frac{1}{q}}$, and $\|(\hat x^*,\hat z^*)\|_{\ga}<\hat\tau$.
Recall from \eqref{psi} that $\psi_y$ is a sum of two functions: the Lipschitz continuous convex function $v\mapsto g_y(v):=\|v-y\|^q$ and the indicator function of the closed set $\gph F$.
Let $\lambda:=\frac{\tau-\hat\tau}{\tau+\hat\tau}$ and choose $\varepsilon>0$  such that
\begin{gather*}
\varepsilon<\min\Big\{\de+\mu-\|x'-\bx\|,\dfrac{1}{2}d(x',F\iv(y)), (\hat\tau\mu)^{\frac{1}{q}}-\|z'-y\|,\\ \hat\tau-\|(\hat x^*,\hat z^*)\|_\gamma,
\min\left\{1/2,(1-\alpha)/8,\lambda\right\}\|z'-y\|\Big\}.
\end{gather*}
Applying Lemma~\ref{SR}(ii), there exist $\hat x\in B_\varepsilon(x')$, and $\hat z,\hat z'\in B_\varepsilon(z')$ with $(\hat x,\hat z)\in\gph F$,  $w^*\in\sd g(\hat z')$, $(u^*,v^*)\in N_{\gph F}(\hat x,\hat z)$ such that
$\|(0,z^*)+(u^*,v^*)-(\hat x^*,\hat z^*)\|_{\ga}<\eps.$
From the choice of $\varepsilon$, one obtain
$\hat x\in B_{\de+\mu}(\bx)$, $\hat x\notin F\iv(y)$, $\|\hat z-y\|<(\hat\tau\mu)^{\frac{1}{q}}$, $\|(0,w^*)+(u^*,v^*)\|_{\ga}<\hat\tau$, and
$\|\hat z-y\|\ge\max\{\frac{1}{2},1-\lambda\}\|z'-y\|$.
It follows from the last estimate that
\begin{gather*}	
\|\hat z'-\hat z\|<\frac{1-\al}{4}\|z'-y\|\le\frac{1-\al}{2} \|\hat z-y\|,\\
\|\hat z'-y\|\le \|\hat z'-\hat z\|+\|\hat z-y\|\le \dfrac{2\lambda}{1-\lambda}\|\hat z-y\|=\dfrac{\tau}{\hat\tau}\|\hat z-y\|.
\end{gather*}
Note that $\hat z'\ne y$, and consequently,
$\partial g(\hat z')=q\|\hat z'-y\|^{q-1}\partial\|\cdot-y\|(\hat z').$
Then, there exists a $z^*\in Y^*$ such that
$w^*=\theta z^*$ with $\theta:=q\|\hat z'-y\|^{q-1}$,
$\|z^*\|=1$ and $\langle z^*,\hat z'-y\rangle=\|\hat z'-y\|$.
One has
\begin{align*}
\langle z^*,\hat z-y\rangle
&\ge\langle z^*,\hat z'-y\rangle-\|\hat z'-\hat z\|=\|\hat z'-y\|-\|\hat z'-\hat z\|\\
&\ge\|\hat z-y\|-2\|\hat z'-\hat z\|>\al\|\hat z-y\|.
\end{align*}
          Let $\hat u^*:=u^*/\theta$, $y^*:=-v^*/\theta$.
Then $(\hat u^*,-y^*)\in N_{\gph F}(\hat x,\hat z)$, and
\begin{gather*}
\|\hat u^*\|+\ga\iv{\|z^*-y^*\|}<\hat\tau q\iv\|\hat z'-y\|^{1-q}.
\end{gather*}	
Then,
\begin{align*}
\|\hat u^*\|+\ga\iv{\|z^*-y^*\|}
&<
\hat\tau q\iv\left(\dfrac{ \tau}{\hat\tau}\right)^{1-q}\|\hat z-y\|^{1-q}\\
&=\hat\tau q\iv\dfrac{\tau}{\hat\tau}\|\hat z-y\|^{1-q}<\tau q\iv \|\hat z-y\|^{1-q}.
\end{align*}	
Then, $q\|\hat z-y\|^{q-1}\|z^*-y^*\|<\eta$ and
$q\|\hat z-y\|^{q-1}\|\hat u^*\|<\tau.$
The last inequality contradicts the assumption.
\item 
Let $\ga:=\tau\iv$,  $x\in B_{\de}(\bx)$, $y\in B_{\de}(\by)$ with $x\notin F\iv(y)$,  and $z\in F(x)$  with $\|z-y\|<\min\{(\tau\mu)^{\frac{1}{q}},1\}$.
Under the assumptions made, the function $\psi_y$ is convex.
For any $(\hat x^*,\hat z^*)\in\sd\psi_y(x,z)$, we have
\begin{align*}
\|(\hat x^*,\hat z^*)\|_{\ga}
&=\sup_{\substack{(u,v)\ne(0,0)}} \dfrac{\ang{(\hat x^*,\hat z^*),(u,v)}} {\|(u,v)\|_{\ga}}\\
&=\limsup_{\substack{u{\to}x,\, v\to z\\
(u,v)\ne(x,z)}} \dfrac{-\ang{(\hat x^*,\hat z^*),(u-x,v-z)}}{\|(u-x,v-z)\|_{\ga}}\\
&\ge\limsup_{\substack{u{\to}x,\, v\to z\\ (u,v)\ne(x,y)}} \dfrac{\psi_y(x,z)-\psi_y(u,v)} {\|(u-x,v-z)\|_{\ga}}\\
&=\limsup_{\substack{u\to x,\,v\to z\\(u,v)\in\gph F,\,(u,v)\ne(x,z)}}
\dfrac{\|z-y\|^q-\|v-y\|^q}{\|(u-x,v-z)\|_\ga}
\ge \tau.
\end{align*}
Observe that $\psi_y$ is the sum of the convex continuous function $v\mapsto g(v):=\|v-y\|^q$ and the indicator function of the convex set $\gph F$.
Note that $z\ne y$, and consequently,
$\partial g(z)=q\|z-y\|^{q-1}\partial\|\cdot-y\|(z).$
By Lemma~\ref{SR}(i), $\sd\psi_y(x,z)=\{0\}\times\sd g(z)+N_{\gph F}(x,z)$.
Let  $\theta:=q\|z-y\|^{q-1}$.
Then, there exist $(u^*,v^*)\in N_{\gph F}(x,z)$ and $z^*\in Y^*$ satisfying $\|z^*\|=1$,  $\langle z^*,z-y\rangle=\|z-y\|$, and
$\|u^*\|+\tau\|\theta z^*+v^*\|\ge \tau.$
Let $\hat u^*:= u^*/\theta$, $ y^*:=-v  ^*/\theta$.
Then $(\hat u^*,-y^*)\in N_{\gph F}(x,z)$, and
$\|\hat u^*\|+\tau\|z^*-y^*\|\ge \tau/\theta.$
Thus, $q\|z-y\|^{q-1}\|\hat u^*\|>\tau(1-\eta)$ if
$q\|z-y\|^{q-1}\|y^*-z^*\|<\eta$
for any $\eta\in (0,1)$.
\end{enumerate}
The proof is complete.
\qed\end{proof}	
\begin{remark}
Parts (i) and (ii) of  Theorem~\ref{P5.6} improve \cite[Theorems~3.1 \& 3.7]{Chu15}, respectively.
In the case $q=1$, Theorem~\ref{P5.6} recaptures \cite[Proposition~5.7]{CuoKru21.4}.
\end{remark}	

\section{The H\"older version of the  Lyusternik-Graves theorem}
\label{S6}

The next theorem establishes the H\"older version of Theorem~\ref{T1.2}.
\begin{theorem}\label{T5.1}
Let $X$ be a complete metric space, $Y$ be a linear  space with a shift-invariant metric, $F:X\rightrightarrows Y$, $(\bx,\by)\in\gph F$, $f: X\rightarrow Y$ with $f(\bx)=0$, $\gph F$ be  closed near 
$(\bx,\by)$, and $q\in(0,1]$.
Then
\begin{align}\label{T5.2-4}
{\rm{rg}}^q(F+f)(\bx,\by)\ge	
{\rm{rg}}^qF(\bx,\by)-{\rm{lip}}^qf(\bx).
\end{align}	
\end{theorem}	

\begin{proof}
If either ${\rm{rg}}^qF(\bx,\by)=0$  or ${\rm{lip}}^{q}f(\bx)=+\infty$, then \eqref{T5.2-4} holds automatically.
Suppose $F$ is metrically regular of order $q$ at $(\bx,\by)$, and $f$ is  H\"older continuous of order $q$ at $\bx$.
Let $\mu:={\rm{lip}}^{q}f(\bx)$.
If  $\mu\ge {\rm{rg}}^qF(\bx,\by)$, then there is nothing to prove.
Suppose ${\mu}< {\rm{rg}}^qF(\bx,\by)$ and let $\tau\in (\mu, {\rm{rg}}^qF(\bx,\by))$.
Choose a $\ga>0$ such that $\tau-\mu>\gamma\iv$.
Under the assumptions made, there exists a $\de>0$  such that the following conditions are satisfied:
\begin{enumerate}
\item[(a)]
inequality 
\eqref{D1.1-1} holds for all $x\in B_{\de}(\bx)$ and $y\in B_{\de}(\by)$;
\item[(b)]
the set $\gph F\cap[\overline B_{\de}(\bx)\times\overline B_{\de}(\by)]$ is closed;
\item[(c)]
$d^q(f(x),f(x'))\le\mu d(x,x')	$ for all $x,x'\in\overline B_{\de}(\bx)$.
\end{enumerate}
Choose a $\hat\de>0$ such that
$\gamma(4\hat\de)^q+\hat\de<\de$,
$\mu^{\frac{1}{q}}(\gamma(4\hat\de)^q+\hat\de)^{{\frac{1}{q}}}+\hat\de<\de$, and $\hat\de^{1-q}<\gamma(2^q-1)$.
Let $x\in B_{\hat\de}(\bx)$ and $y\in B_{\hat\de}(\by)$.
We first show that 
\begin{gather}\label{T1.1-1}
d(x,(F+f)\iv(y))\le \gamma d^q(y,y')
\end{gather}	
for all $y'\in (F(x)+f(x))\cap B_{3\hat\de}(\by)$.
If $y'=y$, then inequality \eqref{T1.1-1} is satisfied.
Suppose $y'\ne y$.
Define $\Phi: X\rightrightarrows X$ by $\Phi(u):=F\iv(-f(u)+y)$ for all $u\in B_{\hat\de}(\bx)$.
We are going to prove that three conditions (i)-(iii) in Lemma~\ref{L2.1}
are satisfied with  $\de':=\ga d^q(y,y')$ and $\theta:=\tau\iv\mu$.
\begin{enumerate}
\item 
Observe that $\de'\le \gamma(d(y,\by)+d(\by,y'))^q< \gamma(4\hat\de)^q$.
Let $\{(x_n,z_n)\}_{n\in\N}$ be a sequence in
$\gph\Phi\cap[\overline B_{\de'}(x)\times \overline B_{\de'}(x)]$ and suppose that it converges to a point $(\hat x,\hat z)\in X\times Y$.
For all $n\in\N$, it holds that
$(z_n,-f(x_n)+y)\in\gph F$,
\begin{gather*}
d(z_n,\bx)\le d(z_n,x)+d(x,\bx)<\de'+\hat\de<\gamma(4\hat\de)^q+\hat\de<\de,
\end{gather*}	
and
\begin{align*}
d(-f(x_n)+y,\by)
&\le d(f(x_n),0)+d(y,\by)
\le \mu^{\frac{1}{q}} d^{{\frac{1}{q}}}(x_n,\bx)+\hat\de\\
&<\mu^{\frac{1}{q}} (\de'+\hat\de)^{{\frac{1}{q}}}+\hat\de<\mu^{\frac{1}{q}} (\gamma(4\hat\de)^q+\hat\de)^{{\frac{1}{q}}}+\hat\de<\de.
\end{align*}	
Thus, $(z_n,-f(x_n)+y)\in\gph F\cap[\overline B_{\de}(\bx)\times\overline B_{\de}(\by)]$ for all $n\in\N$.
Note that $\overline B_{\de'}(x)\subset \overline B_{\de}(\bx)$ since
for any $x'\in \overline B_{\de'}(x)$, one has $d(x',\bx)\le d(x',x)+d(x,\bx)\le\de'+\hat\de<\de.$
The continuity of $f$ on $\overline B_{\de'}(x)$ implies  $(\hat z,-f(\hat x)+y)\in\gph F\cap[\overline B_{\de'}(x)\times\overline B_{\de}(\by)]$.
Thus, $(\hat x,\hat z)\in\gph\Phi\cap[\overline B_{\de'}(x)\times\overline B_{\de'}(x)]$, and consequently,
 $\gph\Phi\cap[\overline{B}_{\de'}(x)\times \overline{B}_{\de'}(x)]$ is closed.
\item 
One has $$d(-f(x)+y,\by)\le d(f(x),0)+d(y,\by)
<\mu^{\frac{1}{q}}\hat\de^{{\frac{1}{q}}}+\hat\de <\de.$$
Hence, $-f(x)+y\in B_{\de}(\by)$.
One has
\begin{align*}
d(x,\Phi(x))
&= d(x,F\iv(-f(x)+y))
\le\tau\iv d^q(-f(x)+y,F(x))\\
&\le\tau\iv d^q(-f(x)+y,-f(x)+y')=
\tau\iv d^q(y,y')\\
&<\gamma d^q(y,y')(1-\tau\iv\mu)=\de'(1-\theta).
\end{align*}	
\item
Observe that
$$d(-f(v )+y,\by)
\le d(f(v),0)+d(y,\by)<\mu^{\frac{1}{q}}\hat\de^{{\frac{1}{q}}}+\hat\de
<\de$$ for any $v\in\overline B_{\de'}(x)$.
For all  $u,v\in\overline B_{\de'}(x)$, one has
\begin{align*}
e(\Phi(u)\cap\overline B_{\de'}(x),\Phi(v))
&=\sup_{z\in F\iv(-f(u)+y)\cap \overline B_{\de'}(x)}d(z,F\iv(-f(v)+y))\\
&\le\tau\iv\sup_{z\in F\iv(-f(u)+y)\cap \overline B_{\de'}(x)} d^q(-f(v)+y,F(z))\\
&\le\tau\iv d^q(f(u),f(v))
\le\tau\iv\mu d(u,v)=\theta d(u,v).
\end{align*}	

\end{enumerate}	
By Lemma~\ref{L2.1}, there exists an
$\hat x\in \overline B_{\de'}(x)$ such that
$\hat x\in\Phi(\hat x)$.
In other words, $y\in F(\hat x)+f(\hat x)$ and $d(\hat{x},x)\le \gamma d^q(y,y')$, and consequently,
\eqref{T1.1-1} holds. 
	
We now arrive at the final stage of showing that
\begin{gather}\label{T1.1-2}
d(x,(F+f)\iv(y))\le 
\gamma d^q(y,F(x)+f(x)).
\end{gather}	
If $F(x)+f(x)=\emptyset$, then there is nothing to prove.
Let $\varepsilon>0$ and choose a point $\hat y\in F(x)+f(x)$ with
$d^q(y,\hat y)< d^q(y,F(x)+f(x))+\varepsilon$.
If $\hat y\in B_{3\hat\de}(\by)$, then	in view of \eqref{T1.1-1}, one obtains
\begin{gather*}
d(x,(F+f)\iv(y))\le\gamma d^q(y,\hat y)\le
\gamma(d^q(y,F(x)+f(x))+\varepsilon).
\end{gather*}	
Letting $\varepsilon\downarrow 0$, one arrives at \eqref{T1.1-2}.
If $\hat y\notin B_{3\hat\de}(\by)$, then
$d(\hat y,y)\ge d(\hat y,\by)-d(y,\by)>2\hat\de,$
and consequently,
\begin{align*}
d(x,(F+f)\iv(y))	
&\le
e(B_{\hat\de}(\bx),(F+f)\iv(y))=\sup_{x'\in B_{\hat\de}(\bx)}d(x',(F+f)\iv(y))\\
&<\hat\de+d(\bx,(F+f)\iv(y))\le {\hat\de}+\tau\iv d^q(y,\by)
<
{\hat\de}+\gamma d^q(y,\by)\\
&<\hat\de+\gamma\hat\de^q\le \gamma(2^q-1)\hat\de^q+\gamma\hat\de^q
=\gamma(2\hat\de)^q\\
&<\gamma d^q(\hat y,y)< \gamma(d^q(y,F(x)+f(x))+\varepsilon).
\end{align*}		 
Letting $\varepsilon\downarrow 0$, one arrives at \eqref{T1.1-2}.
The inequality  \eqref{T5.2-4} is obtained by letting $\gamma\downarrow (\tau-\mu)$.
The proof is complete.
\qed\end{proof}	

\begin{remark}
Theorem~\ref{T5.1} offers an affirmative answer to the question posed by Dontchev \cite[p. 47]{Don15} and enhances \cite[Theorem~3]{HeNg18} providing the exact quantitative estimate for the modulus of regularity of the perturbed mapping.
\end{remark}	

Let $f,g:X\rightarrow Y$ be functions between metric spaces, and $q>0$.
We say that $g$ is a strict approximation of order $q$ to $f$ at $\bx\in X$ if $f(\bx)=g(\bx)$ and ${\rm{lip}}^q(f-g)(\bx)=0$.
In the case $q=1$, the function $g$ is called a strict first-order approximation to $f$ at $\bx$; cf.  \cite[p. 40]{DonRoc14}.

\begin{corollary}\label{C4.1}
Let $X$ be a complete metric space, $Y$ be a linear  space with a shift-invariant metric, $F:X\rightrightarrows Y$, $(\bx,\by)\in\gph F$, $f,g:X\rightarrow Y$ with $f(\bx)=0$, $\gph F$ be closed near $(\bx,\by)$, and $q\in(0,1]$.
Suppose that $g$ is a strict approximation of order $q$ to $f$ at $\bx$.
The mapping $F+f$ is metrically regular of order $q$ at $(\bx,\by)$ if and only if $F+g$ is metrically regular of order $q$ at $(\bx,\by)$.
Moreover, 
\begin{align*}
{\rm{rg}}^q(F+f)(\bx,\by)={\rm{rg}}^q(F+g)(\bx,\by).
\end{align*}	
\end{corollary}	
\begin{proof}
Observe that $F+f=F+g+(f-g)$  and ${\rm{lip}}^{q}(f-g)(\bx)=0$.
The statement is a direct consequence of 
Theorem~\ref{T5.1}.
\qed\end{proof}

In the case  $q=1$, Corollary~\ref{C4.1} yields the following result; cf. \cite[Corollary~5.3]{Don21.2}, \cite[Corollary~2.6]{DonRoc04}.
\begin{corollary}\label{C4.2}
Let $X,Y$ be Banach spaces, $F:X\rightrightarrows Y$, $\gph F$ be closed, $f:X\rightarrow Y$ be continuously Fr\'echet differentiable at $\bx\in X$, and $\by\in F(\bx)+f(\bx)$.
The mapping $F+f$ is metrically regular  at $(\bx,\by)$ if and only if $F(\cdot)+f(\bx)+\nabla f(\bx)(\cdot-\bx)$ is metrically regular at $(\bx,\by)$.
Moreover, 
\begin{align*}
{\rm{rg}}(F+f)(\bx,\by)={\rm{rg}}(F(\cdot)+f(\bx)+\nabla f(\bx)(\cdot-\bx))(\bx,\by).
\end{align*}	
\end{corollary}	

\begin{proof}
Let $g(x):=f(\bx)+\nabla f(\bx)(x-\bx)$ for all $x\in X$.
The continuous differentiability of $f$ implies that
${\rm{lip}}(f-g)(\bx)=0$.
The statement follows from Corollary~\ref{C4.1} with $q=1$.
\qed\end{proof}	


\section{Convergence analysis of a Newton-type method}\label{S5}
Let $X,Y$ be Banach spaces,  $F:X\rightrightarrows Y$, and $f: X\rightarrow Y$.
Consider  the problem 
\begin{center}
	`find $x\in X$ such that $f(x)+F(x)\ni 0$'.
\end{center}	
The aforementioned model has been used to describe in a unified way various problems \cite{DonRoc14,KlaKum02,AleMik14}.
The classical case of a nonlinear equation corresponds to $F(x)=0$, whereas by taking $Y=\R^n$ and $F=\R^n_+$ one has a system of inequality constraints.
The case of $F$ being the normal cone mapping associated with a closed convex subset of a normed space $X$ and $Y=X^*$ results in a variational inequality.

Let $S:=\{x\in X\mid f(x)+F(x)\ni 0\}$ be the solution set.
From now on, we assume that $\bx$ is a given point in $S$.
Consider the Newton sequence  $\{x_k\}_{k\in\N}$ given by
\begin{gather*}
f(x_k)+\nabla f(x_k)(x_{k+1}-x_k)+F(x_{k+1})\ni 0
\;\;
\text{for}
\;\;
k=0,1,2,\ldots.
\end{gather*}

\begin{theorem}\label{T6.2}
Suppose $\gph F$ is closed, $f$ is continuously Fr\'  echet differentiable near $\bx$, the derivative mapping $\nabla f$ is Lipschitz continuous at $\bx$, and $F+f$ is metrically regular  at $(\bx,0)$.
Then there exists a $\de>0$ such that for any $x^\star\in S\cap B_{\frac{\de}{2}}(\bx)$ and $u\in B_{\de}(\bx)$,
there exists a Newton sequence $\{x_k\}_{k\in\N}$ in $B_\de(\bx)$ with $x_0:=u$ converging quadratically to $x^\star$.
\end{theorem}
\begin{proof}
For each $u\in X$, define $\Phi_u: X\rightrightarrows Y$ by  $\Phi_u(x):=f(u)+\nabla f(u)(x-u)+F(x)$ for all $x\in X$.
We first show that there exist $\tau>0$ and $\de'>0$ such that
\begin{gather}\label{phi}
	\tau d(x,\Phi_u\iv(0))\le d(0,\Phi_u(x))	
\end{gather}	
for all $u,x\in B_{\de'}(\bx)$.
By Corollary~\ref{C4.2}, the mapping $\Phi_{\bx}$ is metrically regular at $(\bx,0)$, and ${\rm{rg}}\Phi_{\bx}(\bx,0)={\rm{rg}}(F+f)(\bx,0)$.
Let $\gamma>0$ be such that $\frac{1}{2}{\rm{lip}}\nabla f(\bx)/{\rm{rg}}(F+f)(\bx,0)<\gamma.$
Choose positive numbers $\tau,\tau',\mu,\mu'$ such that
$\tau<\tau'-\mu'<\tau'<{\rm{rg}}(F+f)(\bx,0)$,
$\mu>{\rm{lip}}\nabla f(\bx)$, and $\frac{1}{2}\mu\tau\iv<\gamma.$
The Lipschitz condition guarantees the existence of some constants $\de_1>0$ and $\lambda>0$ such that $\|\nabla f(x)-\nabla f(x')\|\le \lambda \|x-x'\|$ for all $x,x'\in B_{\de_1}(\bx).$
Taking $\de_1$ smaller if necessary, one can ensure that $\Phi_{\bx}$ is metrically regular at $(\bx,0)$ with constants $\tau$ and $\de_1$, and
$\|\nabla f(u)-\nabla f(\bx)\|\le\mu'$ for all $u\in B_{\de_1}(\bx).$
Let $u\in B_{\de_1}(\bx)$.
Define $\psi_u:B_{\de_1}(\bx)\rightarrow Y$   by
\begin{gather*}
\psi_u(x):=f(u)+\nabla f(u)(x-u)-f(\bx)-\nabla f(\bx)(x-\bx)
\;\;
\text{for all}
\;\
x\in B_{\de_1}(\bx).
\end{gather*}	
We obtain
$\|\psi_u(x)-\psi_u(x')\|\le \|\nabla f(u)-\nabla f(\bx)\|\|x-x'\|\le\mu'\|x-x'\|$
for all $x,x'\in B_{\de'}(\bx)$.
Thus, $\psi_u$ is Lipschitz continuous near $\bx$ with ${\rm{lip}}\psi_u(\bx)\le\mu'$.
In view of Theorem~\ref{T5.1} with $q=1$, the mapping
$\Phi_u(x)=\psi_u+\Phi_{\bx}$  is metrically regular at  $(\bx,\psi_u(\bx))$ with constant $\tau'-\mu'$ (and consequently, with constant $\tau$) for some $\de_2\in(0,\de_1]$.
In other words, $\tau d(x,\Phi_u\iv(y))\le d(y,\Phi_u(x))$
for all $x\in B_{\de_2}(\bx)$ and $y\in B_{\de_2}(\psi_u(\bx)).$
Choose a $\de'\in (0,\de_2]$ such that $\frac{\gamma}{2}\de'^2\le\de_2$.
For any $u\in B_{\de'}(\bx)$,
\begin{align*}
\|\psi_u(\bx)\|
&=\|f(u)+\nabla f(u)(\bx-u)-f(\bx)\|\\
&=\left\|\int_{0}^{1}\nabla f(\bx+t(u-x))(u-\bx)dt-\nabla f(u)(u-\bx)\right\|\\ 
&\le\gamma\|u-\bx\|^{2}\int_{0}^{1}(1-t)dt=\dfrac{\gamma}{2}\|u-\bx\|^{2}<\dfrac{\gamma}{2}\de'^2\le\de_2.
\end{align*}	
Thus, \eqref{phi} holds for all $u,x\in B_{\de'}(\bx)$.
Choose $\de\in(0,\de']$ such that
$\frac{9}{2}\gamma \de \le1$, and let $x^\star\in S\cap B_{\frac{\de}{2}}(\bx),$  $u\in B_{\de}(\bx)$.
We are going to prove the existence of $x_1$ satisfying
\begin{gather}\label{T5.3-6}
\Phi_u(x_1)\ni 0,\;\;\;
\|x_1-x^\star\|\le \gamma\|u-x^\star\|^{2},\;\;\;
x_1\in B_{\de}(\bx).
\end{gather}	
If $d(0,\Phi_u(x^\star))=0$, then \eqref{T5.3-6} holds with  $x_1:=x^\star$.
Suppose $d(0,\Phi_u(x^\star))>0$.
One has
\begin{align*}
d(x^\star,\Phi_u\iv(0))
&\le \tau\iv d(0,\Phi_u(x^\star))<2\gamma\mu\iv d(0,\Phi_u(x^\star))\\
&\le 2\gamma\mu\iv \|f(u)+\nabla f(u)(x-u)-f(x^\star)\|\le \gamma\|u-x^\star\|^{2}.
\end{align*}
Thus, there exists an  $x_1\in \Phi_u\iv(0)$ such that
$\|x_1-x^\star\|\le \gamma\|u-x^\star\|^{2}.$
Besides,
\begin{align*}
\|x_1-\bx\|
&\le \|x_1-x^\star\|+\|x^\star-\bx\|<\gamma\|u-x^\star\|^{2}+\dfrac{\de}{2}\\
&\le \gamma (\|u-\bx\|+\|x^\star-\bx\|)^{2}+\dfrac{\de}{2}< \dfrac{9}{4}\gamma \de^{2}+\dfrac{\de}{2}<\dfrac{\de}{2}+\dfrac{\de}{2}=\de.
\end{align*}	
Applying the same argument  with $u:=x_1$, one  obtains the existence of $x_2$ such that
\begin{gather*}
\Phi_u(x_2)\ni 0,\;\;\;
\|x_2-x^\star\|\le \gamma\|x_1-x^\star\|^{2},\;\;\;
x_2\in B_{\de}(\bx).
\end{gather*}	
By this procedure, one can find a Newton sequence $\{x_k\}_{k\in\N}$ in $B_\de(\bx)$ satisfying $\|x_{k+1}-x^\star\|\le\gamma\|x_k-x^\star\|^{2}$ for all $k\in\N$.
Let $\theta:=\gamma\|u-x^\star\|$.
Then,
$\theta\le\gamma(\|u-\bx\|+\|x^\star-\bx\|)
\le\dfrac{3}{2}\de\gamma
\le \dfrac{3}{2}\cdot\dfrac{2}{9}
<1.$
One has $\|x_{k}-x^\star\|\le\theta^{2^{k}-1}\|u-x^\star\|$ for all  $k\in\N$,
and consequently,  $\{x_k\}_{k\in\N}$ converges quadratically  to $x^\star$.
The proof is complete.
\qed\end{proof}	

\begin{remark}
In the particular case  $x^\star=\bx$, Theorem~\ref{T6.2} recaptures \cite[Theorem~15.1]{Don21}.
\end{remark}	

\section{Conclusions}\label{S7}
Primal and dual necessary and sufficient conditions for H\"older metric regularity have been established.
The H\"older version of the extended Lyusternik-Graves theorem providing an affirmative answer to the open question posed by Donchev \cite{Don15} has been proved.
The results have been applied to convergence analysis of a Newton-type method.
The following problems are going to be studied in the future research.

\begin{enumerate}
\item 
Establishing radius results for regularity properties is, of course, an important topic of variational analysis.
However, up to now,  most of the results have been for the linear setting, and there are very few publications studying perturbations of metric regularity properties in the H\"older framework.
Among them, let us mention the papers by
He and Ng \cite{HeNg18} for metric regularity; Mordukhovich and Ouyang \cite{MorOuy15}, Ouyang and Li \cite{OuyLi21}, and Cibulka et al. \cite{CibDonKru18} for strong subregularity.
Hence, it would be good to have a (hopefully) complete picture of stability results under various types of single-valued and set-valued  perturbations in the H\"older setting for properties such as semiregularity \cite{CibFabKru19}; subregularity, strong subregularity, regularity, strong regularity \cite{DonRoc14}; pseudo-regularity \cite{Gfr14}; and others.	
\item 
Formulating parameterized versions of Theorem~\ref{T5.1} for other regularity properties.
In the case of metric regularity, we expect to recapture \cite[Theorem~5E.5]{DonRoc14} when $q=1$.	
\item 
Is the estimate \eqref{T5.2-4}  sharp? In other words, if ${\rm{rg}}^qF(\bx,\by)<+\infty$ and
$\mu\in [0,{\rm{rg}}^qF(\bx,\by)]$, does there exist  a single-valued mapping $f:X\rightarrow Y$ with $f(\bx)=0$, 
\begin{gather*}
{\rm{lip}}^{q}f(\bx)=\mu,
\;\;
\text{and}
\;\;
{\rm{rg}}^q(F+f)(\bx,\by)=	
{\rm{rg}}^qF(\bx,\by)-\mu^q.
\end{gather*}	
In the case $q=1$, the aforementioned problem can be traced back to the open question raised by Ioffe \cite{Iof03} to which Gfrerer and Kruger \cite{GfrKru23} have recently provided an affimative answer in the Asplund setting.
\item
It is well known that metric regularity properties of set-valued mappings have strong connections with transversality properties of collections of sets \cite{CuoKru20,CuoKru21.2,KruTha16}.
While radius results for regularity properties have been investigated, there are no available results for transversality properties.
It would be good to establish radius theorems for models involving collections of sets.
\end{enumerate}	


\noindent{\bf Acknowledgement.}
Nguyen Duy Cuong has been supported by the Postdoctoral Scholarship Programme of Vingroup Innovation Foundation (VinIF) code VINIF.2022.STS.40.
The author wishes to thank Alexander Kruger for comments and suggestions.

\noindent{\bf Conflict of interest.} The author has no competing interests to declare that are relevant to the content of this article.

\noindent{\bf Data availability. }
Data sharing is not applicable to this article as no datasets have been generated or analysed during the current study.

\addcontentsline{toc}{section}{References}
\bibliography{BUCH-kr,Kruger,KR-tmp,Cuong++}
\bibliographystyle{spmpsci}
\end{document}